\documentclass[10pt]{amsart}
\usepackage{amssymb,amsthm,amsmath,amsfonts}
\usepackage{hyperref}
\def\R{{\mathbb R}}
\def\Rr{{\mathcal R}}
\def\C{{\mathbb C}}

\def\Z{{\mathbb Z}}

\def\Tr{\;\mathrm{Tr}\,}
\def\<{\langle}
\def\>{\rangle}
\def\P{{\mathbb P}}

\def\E{{\mathbb E}}
\def\U{{\mathbb U}}
\def\V{{\mathbb V}}
\def\S{{\mathbb S}}

\def\I{\mathbb I}

\def\X{\mathbb X}

\def\Y{\mathbb Y}

\newcommand{\be}{\begin{equation}}
\newcommand{\ee}{\end{equation}}
      \newtheorem{theorem}{Theorem}[section]
       \newtheorem{proposition}[theorem]{Proposition}
       \newtheorem{corollary}[theorem]{Corollary}
       \newtheorem{lemma}[theorem]{Lemma}
       
\title{Dual Lukacs regressions for non-commutative variables}
\author[K. Szpojankowski]{Kamil Szpojankowski}
\address[K. Szpojankowski]{Wydzia\l{} Matematyki i Nauk Informacyjnych\\
Politechnika Warszawska\\
Pl. Politechniki 1\\
00-661 Warszawa, Poland}
\email{k.szpojankowski@mini.pw.edu.pl}

\author[J. Weso\l owski]{Jacek Weso\l owski}
\address[J. Weso\l owski]{Wydzia\l{} Matematyki i Nauk Informacyjnych\\
Politechnika Warszawska\\
Pl. Politechniki 1\\
00-661 Warszawa, Poland}
\email{j.wesolowski@mini.pw.edu.pl}

\subjclass[2010]{Primary: 46L54. Secondary: 62E10.}

\keywords{Lukacs characterization,
conditional moments, freeness, free-Poisson distribution, free-Binomial distribution}

\begin{document}
\begin{abstract}
Dual Lukacs type characterizations of random variables in free probability are studied here. First, we develop a freeness property satisfied by Lukacs type transformations of free-Poisson and free-Binomial non-commutative variables which are free. Second, we give a characterization of non-commutative free-Poisson and free-Binomial variables by properties of first two conditional moments, which mimic Lukacs type assumptions known from  classical probability. More precisely, our result is a non-commutative version of the following result known in classical probability:
if $U$, $V$ are independent real random variables, such that $\E\,(V(1-U)|UV)$ and $\E\,(V^2(1-U)^2|UV)$ are non-random then $V$ has a gamma distribution and $U$ has a beta distribution.
\end{abstract}
\maketitle
\section{Introduction} Characterizations of non-commutative variables and their distributions is a field which develops through non-commutative probability with results which parallel their classical counterparts. It is not completely well understood why the results mirror so much these from the classical setting since the nature of objects under study seems to be much different.

An example of such a result is the Bernstein chracterization of the normal law of independent random variables $X$ and $Y$ by independence of $X+Y$ and $X-Y$ in classical probability \cite{Bernstein} (see also \cite{KaganLinnikRao}), and a characterization of non-commutative semicircular variables $\X$ and $\Y$ which are free and such that $\X+\Y$ and $\X-\Y$ are free by Nica in \cite{NicaChar}.
Similarly, the classical characterization of the normal law by independence of the mean $\bar{X}=\frac{1}{n}\sum_{i=1}^n\,X_i$, and empirical variance $S^2=\frac{1}{n-1}\sum_{i=1}^n\,(X_i-\bar{X})^2$,  where $(X_i)_{i=1,\ldots,n}$ are independent, identically distributed real random variables from \cite{Kawata} is paralleled by a non-comutative characterization of the Wigner law exploiting freeness of $\bar{\X}=\frac{1}{n}\,\sum_{i=1}^n\,\X_i$ and $\S^2=\frac{1}{n-1}\,\sum_{i=1}^n\,(\X_i-\bar{\X})^2$ built on free indentically distributed non-comutative random variables $(\X_i)_{i=1,\ldots,n}$ - see \cite{YoshidaChar}.

In this paper we are concerned with the celebrated Lukacs characterization of the gamma distribution, \cite{Lukacs}. It says that if $X$ and $Y$ are positive, non-degenerate and independent random variables and such that
\be\label{uav}
U=\frac{X}{X+Y}\qquad \mbox{and}\qquad V=X+Y
\ee
are independent then $X$ and $Y$ have gamma distributions, $G(p,a)$ and $G(q,a)$. Here by the gamma distribution $G(r,c)$, $r,c>0$, we understand the probability distribution with density
$$
f(x)=\frac{c^r}{\Gamma(r)}\,x^{r-1}\,e^{-cx}\,I_{(0,\infty)}(x).
$$

The direct result: If $X\sim G(p,a)$ and $Y\sim G(q,a)$ are independent then $U$ and $V$, defined through \eqref{uav}, are independent; is rather simple. It suffices just to compute the jacobian of the bijective transformation $(0,\infty)^2\ni (x,y)\mapsto\left(\tfrac{x}{x+y},x+y\right)\in(0,1)\times(0,\infty)$ and to follow how the densities transform. Immediately it follows also that $V\sim G(p+q,a)$ and $U$ is a beta random variable $B_I(p,q)$, which has the density
$$
f(x)=\frac{\Gamma(p+q)}{\Gamma(p)\,\Gamma(q)}\,x^{p-1}(1-x)^{q-1}\,I_{(0,1)}(x).
$$
The same computation while read backward proves the opposite implication: if $U$ and $V$ are independent, $U\sim G(p+q,a)$ and $V\sim B_I(p,q)$ then
$X=UV$ and $Y=(1-U)V$ are independent, $X\sim G(p,a)$ and $Y\sim G(q,a)$.

For random matrices the role of the gamma law is taken over by Wishart distribution defined, e.g. on the cone $\mathcal{V}_+$ of non-negative definite real $n\times n$ symmetric matrices by the Laplace transform $L({\bf s})=\left(\frac{\det {\bf a}}{\det({\bf a}+{\bf s})}\right)^p$ for positive definite ${\bf a}$ and $p\in\{0,\,\frac{1}{2},\frac{2}{2},\ldots,\frac{n-1}{2}\}\cup\left(\frac{n-1}{2},\,\infty\right)$, and for ${\bf s}$ such that ${\bf a}+{\bf s}$ is positive definite. If $p>\frac{n-1}{2}$ then Wishart distribution has density with respect to the Lebesgue measure on  $\mathcal{V}_+$ of the form
$$
f({\bf x})\propto (\det\,{\bf x})^{p-\frac{n+1}{2}}\,e^{-\Tr\,{\bf a}\,{\bf x}}I_{\mathcal{V}_+}({\bf x}).
$$
Matrix variate beta distribution, in the case of real $n\times n$ matrices, is a probability distribution on the set $\mathcal{D}=\{{\bf x}\in\mathcal{V}_+:\,{\bf I}-{\bf x}\in\mathcal{V}_+\}$ defined by the density
$$
g({\bf x})\propto (\det\,{\bf x})^{p-1}\,(\det\,({\bf I}-{\bf x}))^{q-1},
$$
where the parameters $p,q\in(\frac{n-1}{2},\,\infty)$.

Analogues of Lukacs characterizations have been studied since Olkin and Rubin \cite{OlkinRubin62,OlkinRubin64}. In particular, Casalis and Letac \cite{CasalisLetac96} obtained such a characterization in a general setting of probability measures on symmetric cones (including positive definite real symmetric and hermitian matrices), though they assumed an additional structural invariance property.  Under smoothness conditions on densities Bobecka and Weso\l owski \cite{BobWes2002} proved that if ${\bf X}$ and ${\bf Y}$ are independent $\mathcal{V}_+$-valued random matrices and ${\bf U}=({\bf X}+{\bf Y})^{-\frac{1}{2}}\,{\bf X}\,({\bf X}+{\bf Y})^{-\frac{1}{2}}$ and ${\bf V}={\bf X}+{\bf Y}$ are independent then ${\bf X}$ and ${\bf Y}$ are Wishart matrices. For recent extensions see \cite{Boutoria2009,BoutHassMass,Kolodziejek}.

In the context of Lukacs type characterizations of distributions of random variables in non-commutative setting the analogue of gamma distribution is free-Poisson (Marchenko-Pastur) distribution (Note that this analogy follows neither Berkovici-Pata bijection \cite{BerkoviciPata}, nor analogy between classical and free Meixner distributions defined in \cite{anshelevich} and developed in \cite{BoBr2006} i.e. it is not free gammma distribution). Let $\X$ and $\Y$ be free non-commutative variables having free-Poisson distributions (with properly defined parameters). Define
$$\U=(\X+\Y)^{-1/2}\,\X(\X+\Y)^{-1/2}\qquad\mbox{and}\qquad \V=\X+\Y.$$
One would suspect that by the analogy to the classical case or to the matrix variate situation, $\U$ and $\V$ are free. This is still an open problem. The closest result has been derived in \cite{CaCa}(referred to by CC in the sequel). They proved that for complex Wishart independent matrices ${\bf X}$ and ${\bf Y}$ the matrices ${\bf U}$ and ${\bf V}$ defined as for the real case above are asymptotically free and the limiting (in non-commutative sense) distributions of ${\bf U}$ and ${\bf V}$ were derived to be free-Poisson and a distribution which, by the analogy to classical (univariate or matrix-variate) cases could be called free-beta, but it has already been known under the name free-binomial (for details, see Sect. 7 in CC; consult also the first part of Section 3 below where a complete description of the set of parameters of this distribution is presented). The main result of Section 3 goes in the opposite direction. Assuming that free variables $\U$ and $\V$ have, respectively, free-binomial and free-Poisson distributions we prove that $\X$ and $\Y$ are free with suitable free-Poisson distribution each. This is done through developing some ideas from CC.

The direct non-commutative version of Lukacs characterization, saying that if $\X$ and $\Y$ are free and $\U$ and $\V$, as defined above, are free then $\X$ and $\Y$ are free-Poisson was obtained in \cite{BoBr2006}, Prop. 3.5.

The classical Lukacs characterization can be obtained with weaker assumptions than independence of $U$ and $V$. Such assumptions may be  formulated in the language of constancy of regressions. For instance, it is known that if $X$ and $Y$ are positive, non-degenerate and independent and
\begin{equation}\label{abab}
\E(X|X+Y)=c(X+Y)\qquad\mbox{and}\qquad \E(X^2|X+Y)=d(X+Y)^2
\end{equation}
for some real numbers $c$ and $d$ then $X$ and $Y$ are necessarily gamma distributed, $G(p,a)$ and $G(q,a)$, where the parameters $p$ and $q$ depend on $c$ and $d$. This can be traced back to Bolger and Harkness, \cite{BolHark}. But the result is also hidden as one of special cases in the celebrated Laha and Lukacs paper \cite{LahaLukacs}. Regression versions of Lukacs type characterizations of Wishart random matrices were obtained in \cite{LetacMassam} and more recently in the framework of regressions of quadratic forms in \cite{LetacWes2008,LetacWes2011}.

The non-commutative version can be found in \cite{BoBr2006}, which is devoted mostly to Laha-Lukacs type characterizations. The authors assumed that $\varphi(\X|\X+\Y)=\frac{1}{2}(\X+\Y)$ which is an analogue of first part of \eqref{abab} for identically distributed $X$ and $Y$, but instead of a direct non-commutative version of the second part of \eqref{abab} they considered, following the classical setting of \cite{LahaLukacs}, a more general condition
$$
\varphi(\X^2|\X+\Y)=a(\X+\Y)^2+b(\X+\Y)+c\I.
$$
They used a cumulant approach to derive possible distributions of $\X$ and $\Y$. A related result in converse direction for free-Poisson variables has been given recently in \cite{Ejsmont}

Our aim here is to consider the dual regression scheme of Lukacs type. In the classical setting it means that we take idependent $U$ which is $(0,1)$-valued and $V$ which is positive and assume that
$$
\E((1-U)V|UV)=c\qquad\mbox{and}\qquad\E((1-U)^2V^2|UV)=d
$$
for some constants $c$ and $d$. It was proved in \cite{BobWes2002Dual} that then necessarily $U$ and $V$ are, respectively, beta and gamma random variables (see also \cite{ChouHuang} for a more general characterization). A version of this characterization in the non-cummutative setting which is considered in Section 4, is the main result of this paper.

Next section is devoted to basics of non-commutative probability we need to explain the results and derivations.

\section{Preliminaries} Following \cite{VoiDykNica} or \cite{NicaSpeicherLect} we will recall basic notions of non-commutative probability which are necessary for this paper.

A non-commutative probability space is a pair $(\mathcal{A},\varphi)$, where $\mathcal{A}$ is a unital algebra over $\C\,$ and $\varphi:\mathcal{A}\to\C$ is a linear functional satisfying $\phi(\I)=1$. Any element $\X$ of $\mathcal{A}$ is called a (non-commutative) random variable.

Let $H$ be a Hilbert space. By $\mathcal{B}(H)$ denote the space of bounded linear operators on $H$. For $\mathcal{A}\subset\mathcal{B}(H)$ and $\varphi\in H$ we say that $(\mathcal{A},\,\varphi)$ is a $W^*$-probability space when $\mathcal{A}$ is a von Neumann algebra and $\varphi$ is a normalized, faithfull and tracial state, that is $||\varphi||=1$, $\varphi(\X^2)=\langle\X^2\varphi,\varphi\rangle=0$ iff $\X=0$ and $\varphi(\X\,\Y)=\varphi(\Y^*\,\X)$ for any $\X,\Y\in\mathcal{B}(H)$.

The $*$-distribution $\mu$ of a self-adjoint element $\X\in\mathcal{A}\subset\mathcal{B}(H)$ is a probabilistic measure on $\R$ such that $$\varphi(\X^r)=\int_{\R}\,t^r\,\mu(dt)\qquad \forall\,r=1,2,\ldots$$ In a setting of a general non-commutative probability space $(\mathcal{A},\varphi)$, we say that the distribution of the family $(\X_i)_{i=1,\ldots,q}$ is a linear functional $\mu_{\X_1,\ldots,\X_q}$ on the algebra $\C\,\langle x_1,\ldots,x_q\rangle$ of polynomials of non-commuting variables $x_1,\ldots,x_q$, defined by $$\mu_{\X_1,\ldots,\X_q}(P)=\varphi(P(\X_1,\ldots,\X_q))\qquad\forall\,P\in\C\,\langle x_1,\ldots,x_q\rangle.$$

Unital subalgebras $\mathcal{A}_i\subset \mathcal{A}$, $i=1,\ldots,n$, are said to be freely independent if $\varphi(\X_1,\ldots,\X_k)=0$ for $\X_j\in \mathcal{A}_{i(j)}$, where $i(j)\in\{1,\ldots,n\}$, such that $\varphi(\X_j)=0$, $j=1,\ldots,k$, if neighbouring elements are from different subalgebras, that is $i(1)\ne i(2)\ne \ldots \ne i(k-1)\ne i(k)$. Similarly, random variables $\X,\,\Y\in\mathcal{A}$ are free (freely independent) when subalgebras generated by $(\X,\,\I)$ and $(\Y,\,\I)$ are freely independent (here $\I$ denotes identity operator).

For free random variables $\X$ and $\Y$ having distributions $\mu$ and $\nu$, respectively, the distribution of $\X+\Y$, denoted by $\mu\boxplus\nu$, is called free convolution of $\mu$ and $\nu$.

For self-adjoint and free $\X$, $\Y$ with distributions $\mu$ and $\nu$, respectively, and $\X$ positive, that is the support of $\mu$ is a subset of $(0,\infty)$, free multiplicative convolution of $\mu$ and $\nu$ is defined as the distribution of $\sqrt{\X}\,\Y\sqrt{\X}$ and denoted by $\mu\boxtimes\nu$. Due to the tracial property of $\varphi$ the moments of $\Y\,\X$, $\X\,\Y$ and $\sqrt{\X}\,\Y\sqrt{\X}$ match.

Let $\chi=\{B_1,B_2,\ldots\}$ be a  partition of the set of numbers $\{1,\ldots,k\}$. A partition $\chi$ is a crossing partition if there exist distinct blocks $B_r,\,B_s\in\chi$ and numbers $i_1,i_2\in B_r$, $j_1,j_2\in B_s$ such that $i_1<j_1<i_2<j_2$. Otherwise $\chi$ is called a non-crossing partition. The set of all non-crossing partitions of $\{1,\ldots,k\}$ is denoted by $NC(k)$.

For any $k=1,2,\ldots$, (joint) cumulants of order $k$ of non-commutative random variables $\X_1,\ldots,\X_n$ are defined recursively as $k$-linear maps $\mathcal{R}_k:\C\,\langle x_i,\,i=1,\ldots,k\rangle\to\C$ through equations
$$
\varphi(\Y_1,\ldots,\Y_m)=\sum_{\chi\in NC(m)}\,\prod_{B\in\chi}\,\mathcal{R}_{|B|}(x_i,\,i\in B)
$$
holding for any $\Y_i\in\{\X_1,\ldots,\X_n\}$, $i=1,\ldots,m$, and any $m=1,2,\ldots$,
with $|B|$ denoting the size of the block $B$.

Freeness can be characterized in terms of behaviour of cumulants in the following way: Consider unital subalgebras $(\mathcal{A}_i)_{i\in I}$ of an algebra $\mathcal{A}$ in a non-commutative probability space $(\mathcal{A},\,\varphi)$. Subalgebras $(\mathcal{A}_i)_{i\in I}$ are freely independent iff for any $n=2,3,\ldots$ and for any $\X_j\in\mathcal{A}_{i(j)}$ with $i(j)\in I$, $j=1,\ldots,n$ any $n$-cumulant
$$
\mathcal{R}_n(\X_1,\ldots,\X_n)=0
$$
if there exists a pair $k,l\in\{1,\ldots,n\}$ such that $i(k)\ne i(l)$.

In sequel we will use the following formula from \cite{BozLeinSpeich} which connects cumulants and moments for non-commutative random variables
\be\label{BLS}
\varphi(\X_1\ldots\X_n)=\sum_{k=1}^n\,\sum_{1<i_2<\ldots<i_k\le n}\,\mathcal{R}_k(\X_1,\X_{i_2},\ldots,\X_{i_k})\,\prod_{j=1}^k\,\varphi(\X_{i_j+1}\ldots\X_{i_{j+1}-1})
\ee
with $i_1=1$ and $i_{k+1}=n+1$ (empty products are equal 1).

The classical notion of conditional expectation has its non-commutative counterpart in the case $(\mathcal{A},\varphi)$ is a  $W^*$-probability spaces, that is $\mathcal{A}$ is necessarily a von Neumann algebra. Namely, if $\mathcal{B}\subset \mathcal{A}$ is a von Neumann subalgebra of the von Nuemann algebra $\mathcal{A}$, then there exists a faithful normal projection from $\mathcal{A}$ onto $\mathcal{B}$, denoted by $\varphi(\cdot|\mathcal{B})$, such that $\varphi(\varphi(\cdot|\mathcal{B}))=\varphi$. This projection $\varphi(\cdot|\mathcal{B})$ is a non-commutative conditional expectation given subalgebra $\mathcal{B}$. If $\X\in \mathcal{A}$ is self-adjoint then $\varphi(\X|\mathcal{B})$ defines a unique self-adjoint element in $\mathcal{B}$. For $\X\in\mathcal{A}$ by $\varphi(\cdot|\X)$ we denote conditional expectation given von Neumann subalgebra $\mathcal{B}$ generated by $\X$ and $\I$. Non-commutative conditional expectation has many properties analogous to those of classical conditional expectation. For more details one can consult e.g. \cite{Takesaki}. Here we state two of them we need in the sequel. The proofs can be found in \cite{BoBr2006}.
\begin{lemma}\label{conexp} Consider a $W^*$-probability space $(\mathcal{A},\varphi)$.
\begin{itemize}
\item If $\X\in\mathcal{A}$ and $\Y\in\mathcal{B}$, where $\mathcal{B}$ is a von Neumann subalgebra of $\mathcal{A}$, then
\be\label{ce1}
\varphi(\X\,\Y)=\varphi(\varphi(\X|\mathcal{B})\,\Y).
\ee
\item If $\X,\,\Z\in\mathcal{A}$ are freely independent then
\be\label{ce2}
\varphi(\X|\Z)=\varphi(\X)\,\mathbb{I}.
\ee
\end{itemize}
\end{lemma}

For any $n=1,2,\ldots$, let $(\X_1^{(n)},\ldots,\X_q^{(n)})$ be a family of random variables in a non-commutative probability space $(\mathcal{A}_n,\varphi_n)$. The sequence of distributions $(\mu_{(\X_i^{(n)},\,i=1,\ldots,q)})$ converges as $n\to\infty$ to a distribution $\mu$ if $\mu_{(\X_i^{(n)},\,i=1,\ldots,q)}(P)\to \mu(P)$ for any $P\in\C\,\langle x_1,\ldots,x_q\rangle$. If additionally $\mu$ is a distribution of a family $(\X_1,\ldots,\X_q)$ of random variables in a non-commutative space $(\mathcal{A},\varphi)$ then we say that $(\X_1^{(n)},\ldots,\X_q^{(n)})$ converges in distribution to $(\X_1,\ldots,\X_q)$. Moreover, if $\X_1,\,\ldots,\X_q$ are freely independent then we say that $\X_1^{(n)},\ldots,\X_q^{(n)}$ are asymptotically free.

Now we introduce basic analytical tools used to deal with non-commutative random variables and their distributions.

For a non-commutative random variable $\X$ its $r$-transform is defined as
\be\label{rtr}
r_{\X}(z)=\sum_{n=0}^{\infty}\,\mathcal{R}_{n+1}(\X)\,z^n.
\ee
In \cite{VoiculescuAdd} it is proved that $r$-transform of a random variable with compact support is analytic in a neighbourhood of zero.  From properties of cumulants it is immediate that for $\X$ and $\Y$ which are freely independent
\be\label{freeconv}
r_{\X+\Y}=r_{\X}+r_{\Y}.
\ee
This relation explicitly (in the sense of $r$-transform) defines free convolution of $\X$ and $\Y$.
If $\X$ has the distribution $\mu$, then often we will write $r_{\mu}$ instead $r_{\X}$.

Another analytical tool is an $S$-transform, which works nicely with products of freely independent variables. For a noncommutative random variable $\X$ its $S$-transform, denoted by $S_{\X}$, is defined through the equation
\be\label{Str}
R_{\X}(zS_{\X}(z))=z,
\ee
where $R_{\X}(z)=zr_{\X}(z)$. For $\X$ and $\Y$ which are freely independent
\be\label{Scon}
S_{\X\,\Y}=S_{\X}\,S_{\Y}.
\ee

Cauchy transform of a probability measure $\nu$ is defined as
$$
G_{\nu}(z)=\int_{\R}\,\frac{\nu(dx)}{z-x},\qquad \Im(z)>0.
$$
Cauchy transforms and $r$-transforms are related by
\be\label{Crr}
G_{\nu}\left(r_{\nu}(z)+\frac{1}{z}\right)=z.
\ee

Finally we introduce moment generating function $M_{\X}$ of a random variable $\X$ by
\be\label{mgf}
M_{\X}(z)=\sum_{n=1}^{\infty}\,\varphi(\X^n)\,z^n.
\ee
Moment generating function and $S$-transform of $\X$ are related through
\be\label{MSr}
M_{\X}\left(\frac{z}{1+z}\,S_{\X}(z)\right)=z.
\ee

\section{Free transformations of freely independent free-Poisson and free-binomial variables}
A non-commutative random variable $\X$ is said to be free-Poisson variable if it has Marchenko-Pastur(or free-Poisson) distribution $\nu=\nu(\lambda,\alpha)$ defined by the formula
\be\label{MPdist}
\nu=\max\{0,\,1-\lambda\}\,\delta_0+\lambda \tilde{\nu},
\ee
where $\lambda\ge 0$ and the measure $\tilde{\nu}$, supported on the interval $(\alpha(1-\sqrt{\lambda})^2,\,\alpha(1+\sqrt{\lambda})^2)$, $\alpha>0$ has the density (with respect to the Lebesgue measure)
$$
\tilde{\nu}(dx)=\frac{1}{2\pi\alpha x}\,\sqrt{4\lambda\alpha^2-(x-\alpha(1+\lambda))^2}\,dx.
$$
The parameters $\lambda$ and $\alpha$ are called the rate and the jump size, respectively.

Marchenko-Pastur distribution arises in a natural way  as an almost sure weak limit of empirical distributions of eigenvalues for random matrices of the form ${\bf X}\,{\bf X}^T$ where ${\bf X}$ is a matrix with zero mean iid entries with finite variance, in particular for Wishart matrices, (see \cite{MarchPastur}) and as a marginal distribution of a subclass of classical stochastic processes, called quadratic harnesses (see e.g. \cite{BrycWes2005}).

It is worth to note that a non-commutative variable with Marchenko-Pastur distribution arises as a limit in law (in non-commutative sense) of variables with distributions $((1-\frac{\lambda}{N})\delta_0+\frac{\lambda}{N}\delta_{\alpha})^{\boxplus N}$ as $N\to\infty$, see \cite{NicaSpeicherLect}. Therefore, such variables are often called free-Poisson.

It is easy to see that if $\X$ is free-Poisson, $\nu(\lambda,\alpha)$, then $\mathcal{R}_n(\X)=\alpha^n\lambda$, $n=1,2,\ldots$. Therefore its $r$-transform has the form
$$
r_{\nu(\lambda,\alpha)}(z)=\frac{\lambda\alpha}{1-\alpha z}.
$$

A non-commutative random variable $\Y$ is free-binomial if its distribution $\beta=\beta(\sigma,\theta)$ is defined by
\be\label{freebeta}
\beta=(1-\sigma)\mathbb{I}_{0<\sigma<1}\,\delta_0+\tilde{\beta}+(1-\theta)\mathbb{I}_{0<\theta<1}\delta_1,
\ee
where $\tilde{\beta}$ is supported on the interval $(x_-,\,x_+)$,
$$
x_{\pm}=\left(\sqrt{\frac{\sigma}{\sigma+\theta}\,\left(1-\frac{1}{\sigma+\theta}\right)}\,\pm\,\sqrt{\frac{1}{\sigma+\theta}\left(1-\frac{\sigma}{\sigma+\theta}\right)}\right)^2,
$$
has the density
$$
\tilde{\beta}(dx)=(\sigma+\theta)\,\frac{\sqrt{(x-x_-)\,(x_+-x)}}{2\pi x(1-x)}\,dx.
$$

This distributions appears in CC (unfortunately, the constant $\alpha+\beta$ is missing in the expression given for the density part in Cor. 7.2 in CC) as a limit distribution for beta matrices as well as spectral distribution for free Jacobi processes in \cite{Demni2008} and for a subclass of free quadratic harnesses (see \cite{FreeHarness} and \cite{BrycProcFreeMeix}).
The n-th free convolution power of distribution
$$
p\delta_0+(1-p)\delta_{1/n}
$$
is free-binomial distribution with parameters $\sigma=n(1-p)$ and $\theta=np$, which justifies the name of the distribution (see \cite{SaitohYosida}).

Its Cauchy transform is of the form (see e.g. the proof of Cor. 7.2 in CC)
\be\label{betaC}
G_{\sigma,\theta}(z)=\frac{(\sigma+\theta-2)z+1-\sigma-\sqrt{[(\sigma+\theta-2)z+1-\sigma]^2-4(1-\sigma-\theta)z(z-1)}}{2z(1-z)}.
\ee
So far the range of parameters $\sigma,\theta$ for which \eqref{freebeta} is a true probability distribution has not been completely described in the literature. In CC  the authors seem to assume $\theta, \sigma>0$, which apparently is not enough for a correct definition. On the other hand \cite{Demni2008} assumes $\sigma,\theta>1$ which is too restrictive in general. The complete set of parameters is described below.
\begin{proposition}
The formula \eqref{freebeta} defines correctly a probability (free-binomial) distribution if and only if
$(\sigma,\theta)\in G$ defined as
$$
G=\left\{(\sigma,\theta):\,\frac{\sigma+\theta}{\sigma+\theta-1}>0,\,\frac{\sigma\theta}{\sigma+\theta-1}>0\right\}.
$$
\end{proposition}
\begin{proof}
Recall a result from \cite{BrycProcFreeMeix} which says that the two parameters Askey-Wilson probability measure, which has the form
\be
\label{AW}
\nu(dx)=\frac{2(1-ab)}{\pi}\frac{\sqrt{1-x^2}}{(1+a^2-2ax)(1+b^2-2bx)}dx+
\frac{a^2-1}{a^2-ab}\mathbb{I}_{|a|>1}\delta_\frac{a+1/a}{2}+
\frac{b^2-1}{b^2-ab}\mathbb{I}_{|b|>1}\delta_\frac{b+1/b}{2},
\ee
is well defined iff $ab<1$. With an additional natural  assumption $ab\neq 0$, this probability law can be easily transformed into a free-binomial distribution. Indeed, if we take a random variable $X$ with the above Askey-Wilson distribution, then $Y$ defined as
\be
Y=\frac{a}{(a-b)(ab-1)}\left(2bX-(1+b^2)\right),
\ee
has a free-binomial distribution \eqref{freebeta} with parameters $$(\theta,\,\sigma)=\psi(a,b)=\left(\frac{1-ab}{a(a-b)},\,\frac{1-ab}{b(b-a)}\right).$$

Define now
\be
H=\{(a,b):\,1-ab>0,ab\ne 0\}.
\ee
It is rather immediate to see that if $(a,b)\in H$ then $\psi(a,b)\in G$.

Conversely, define an equivalence relation on $H$ by $(a,b)\sim(-a,-b)$ and note that $\psi$ is a bijection between $H/_\sim$ and  $G$.

Finally, referring to the result in \cite{BrycProcFreeMeix} on Askey-Wilson distributions mentioned above, we conclude that \eqref{freebeta} defines correctly a probability measure (free-binomial distribution) iff $(\sigma,\theta)\in G$.
\end{proof}

Note that in \cite{SaitohYosida} the parameters of free-binomial distributions are $\sigma=n(1-p)$ $\theta=np$, and $n\geq 2$ so above conditions are satisfied. Nevertheless, their parametrization does not cover whole $G$, e.g. the situation when one of the parameters is negative, which is allowed. It is worth to note that our derivation extends free-binomial distribution to the case $\sigma+\theta<0$, in this case continuous part of free-binomial is not supported in $(0,1)$, in case $\sigma<0$ continuous part is supported on $(-\infty,-1)$, in case $\theta<0$ continuous part is supported on $(1,\infty)$.

In CC authors consider complex Wishart matrices $N\times N$, corresponding Gindikin set is $\{1,2,\dots,N-2\}\cup[N-1,\infty)$ (see \cite{PeddadaRichards1991}). To define a beta matrix as $\bf Z=(X+Y)^{-1/2}X(X+Y)^{-1/2}$, where $\bf X, Y$ are independent Wishart matrices, matrix $\bf X+Y$ must be invertible, so one has to assume $p_N+q_N>N-1$. For existence of non-degenerate asymptotic distributions it is necessary that $p_N/N\to\sigma>0$ and $q_N/N\to\theta>0$. Moreover, either $\sigma+\theta-1>0$ and $(\sigma,\theta)\in G$, or $\sigma+\theta=1$ and the limit distribution (in non commutative sense) of beta matrices has only discret part. In the sequel we are concerned only with case $(\sigma,\theta)\in G$ and $\sigma,\theta>0$.

The main result of this section, as announced in Introduction, is a direct dual version of Lukacs characterization.
\begin{theorem}\label{direc}
Let $(\mathcal{A},\,\varphi)$ be a $W^*$-probability space. Let $\V$ and $\U$ in $\mathcal{A}$ be freely independent, such that $\V$ is free-Poisson with parameters $(\lambda,\,\alpha)$ and $\U$ is free-binomial with parameters $(\sigma,\theta)$, $\sigma+\theta=\lambda$. Define
\be\label{XY}
\X=\V^{\frac{1}{2}}\,\U\,\V^{\frac{1}{2}}\qquad\mbox{and}\qquad\Y=\V-\V^{\frac{1}{2}}\,\U\,\V^{\frac{1}{2}}.
\ee

Then $\X$ and $\Y$ are freely independent and their distributions are free-Poisson with parameters $(\theta,\alpha)$  and $(\sigma,\alpha)$, respectively.
\end{theorem}
Throughout this paper we use the framework of $W^*$-probability space, since essentially we work with conditional expectations, however the above theorem holds true in a more general setting when  $(\mathcal{A},\,\varphi)$ is a $C^*$-probability space with tracial state (see \cite[Chapter 3]{NicaSpeicherLect}) with exactly the same proof.
\begin{proof}
Since freeness is defined for subalgebras then, without any loss of generality, instead of $\V$ we can take $\frac{1}{\alpha}\V$ which is free-Poisson with the jump size equal 1.
Consider a non-commutative probability space $(\mathcal{A}_N,\,\varphi_N)$ of $p$-integrable for any $1\le p<\infty$ random matrices of dimension $N\times N$ defined on a classical probability space $(\Omega,\mathcal{F},\P)$ with $\varphi_N({\bf A})=\frac{1}{N}\,\E\Tr{\bf A}$ for any ${\bf A}\in\mathcal{A}_N$. From the proof of Th. 5.2 and Prop. 5.1 in CC, in which  asymptotic freeness of $\U$ and $\V$ is stated, it follows that there exist independent sequences $({\bf U}_N)_N$ and $({\bf V}_N)_N$ of complex $N\times N$ matrices  such  that $({\bf U}_N,\,{\bf V}_N)$ converges in distribution in the non-commutative sense (as elements of non-commutative probability spaces $(\mathcal{A}_N,\varphi_N)$) to $(\U,\V)$ which are freely independent. Moreover, ${\bf U}_N$ is a beta matrix with suitable positive parameters $p_N,\,q_N$ such that $\frac{p_N}{N}\to \sigma$, $\frac{q_N}{N}\to\theta$ and ${\bf V}_N$ is a Wishart matrix with parameters $p_N+q_N$ and $\frac{1}{N}{\bf I}_N$, where ${\bf I}_N$ is an $N\times N$ identity matrix. It is well known in such case, see e.g. \cite{CasalisLetac96}, that $${\bf X}_N={\bf V}^{\frac{1}{2}}_N\,{\bf U}_N\,{\bf V}^{\frac{1}{2}}_N\qquad\mbox{and}\qquad{\bf Y}_N={\bf V}_N-{\bf V}^{\frac{1}{2}}_N\,{\bf U}_N\,{\bf V}^{\frac{1}{2}}_N$$ are independent complex Wishart matrices with parameters $(p_N,\frac{1}{N}{\bf I}_N)$ and $(q_N,\frac{1}{N}{\bf I}_N)$, respectively. By Th. 5.2 from CC it follows that $({\bf X}_N,\,{\bf Y}_N)$ are asymptotically free. Therefore, see Prop. 4.6 in CC, it follows that there exist freely independent non-commutative variables $\X'$ and $\Y'$ with free-Poisson distributions with jump parameter 1 and rate parameters $\sigma$ and $\theta$, respectively, such that $({\bf X}_N,{\bf Y}_N)$ converges in distribution (in the non-commutative sense) to $(\X',\,\Y')$.

By asymptotic freeness it follows that
\be\label{asmom}
\lim_{N\to\infty}\,\varphi_N(P({\bf X}_N,{\bf Y}_N))=\varphi(P(\X',\,\Y'))
\ee
for an arbitrary non-commutative polynomial $P\in\C\,\langle x_1,x_2\rangle$. On the other hand by the definition of ${\bf X}_N$ and ${\bf Y}_N$
$$
\varphi_N(P({\bf X}_N,\,{\bf Y}_N))=\varphi_N(P({\bf V}_N-{\bf V}_N^{1/2}{\bf U}_N{\bf V}_N^{1/2},\;{\bf V}_N^{1/2}{\bf U}_N{\bf V}_N^{1/2})).
$$
By the tracial property of $\varphi_N$ the last expression can be written as
$$
\varphi_N(Q({\bf U}_N,\,{\bf V}_N)),
$$
for some polynomial $Q$ from $\C\,\langle x_1,x_2\rangle$. Since $({\bf U}_N,\,{\bf V}_N)$ converge in distribution (in non-commutative sense) to $(\U,\,\V)$ it follows that
$$
\lim_{N\to\infty}\,\varphi_N(P({\bf X}_N,{\bf Y}_N))=\lim_{N\to\infty}\,\varphi_N(Q({\bf U}_N,\,{\bf V}_N))=\varphi(Q(\U,\,\V)).
$$
Using the tracial property of $\varphi$ we can return from $Q$ to $P$, so that
$$
\varphi(Q(\U,\,\V))=\varphi(P(\V-\V^{1/2}\,\U\V^{1/2},\,\V^{1/2}\,\U\V^{1/2}))=\varphi(P(\X,\,\Y)).
$$
Therefore \eqref{asmom} implies that for any $P\in\C\,\langle x_1,x_2\rangle$
$$
\varphi(P(\X',\,\Y'))=\varphi(P(\X,\,\Y)).
$$
Consequently, they have the same distribution.
\end{proof}

\begin{corollary}
\label{coro}
Let $\U$ and $\V$ be freely independent random variables in a $W^*$-probability space. Assume that $\V$ is free-Poisson with parameters $\theta+\sigma$ and $\alpha$ and $\U$ is free-binomial with parameters $\sigma$ and $\theta$. Then
$$
\varphi\left(\left.\V-\V^{\frac{1}{2}}\,\U\,\V^{\frac{1}{2}}\right|\V^{\frac{1}{2}}\,\U\,\V^{\frac{1}{2}}\right)=\theta\alpha\,\I
$$
and
$$
\varphi\left(\left.\left(\V-\V^{\frac{1}{2}}\,\U\,\V^{\frac{1}{2}}\right)^2\right|\V^{\frac{1}{2}}\,\U\,\V^{\frac{1}{2}}\right)=\theta(\theta+1)\alpha^2\,\I.
$$
\end{corollary}

\begin{proof}
By Theorem \ref{direc} we know that $\X=\V^{\frac{1}{2}}\,\U\,\V^{\frac{1}{2}}$ and $\Y=\V-\V^{\frac{1}{2}}\,\U\,\V^{\frac{1}{2}}$ are freely independent. Therefore \eqref{ce2} of Lemma \ref{conexp} implies
$$
\varphi(\X|\Y)=\varphi(\X)\,\I
$$
and
$$
\varphi(\X^2|\Y)=\varphi(\X^2)\,\I.
$$
Due to Theorem \ref{direc} $\X$ is free-Poisson with parameters $\theta$, $\alpha$. It is well known that for free-Poisson $\X$ its first two moments are $\varphi(\X)=\theta\alpha$ and $\varphi(\X^2)=\theta(\theta+1)\alpha^2$.
\end{proof}

\section{Non-commutative dual Lukacs type regression}
In this section we formulate and prove the main result of the paper which may be treated as a counterpart of the regression characterization of free-Poisson distribution given in Th. 3.2 (ii) of \cite{BoBr2006}. It is also a non-commutative version of Th. 1 of \cite{BobWes2002Dual}, and a converse to Cor. \ref{coro} above.
\begin{theorem}\label{main}
Let $(\mathcal{A},\,\varphi)$ be a $W^*$-probability space and $\U,\,\V$ be non-commutative variables in $(\mathcal{A},\,\varphi)$ which are freely independent, $\V$ has a distribution compactly supported in $(0,\infty)$ and  distribution of $\U$ is supported in $[0,1]$. Assume that there exist real constants $c_1$ and $c_2$ such that
\be\label{reg1}
\varphi\left(\left.\V-\V^{\frac{1}{2}}\,\U\,\V^{\frac{1}{2}}\right|\V^{\frac{1}{2}}\,\U\,\V^{\frac{1}{2}}\right)=c_1\,\I
\ee
and
\be\label{reg2}
\varphi\left(\left(\left.\V-\V^{\frac{1}{2}}\,\U\,\V^{\frac{1}{2}}\right)^2\right|\V^{\frac{1}{2}}\,\U\,\V^{\frac{1}{2}}\right)=c_2\,\I.
\ee

Then $\V$ has free-Poisson distribution, $\nu(\lambda,\alpha)$ with $\lambda=\sigma+\theta$, $\left(\sigma>0,\theta=\frac{c_1^2}{c_2-c_1^2}>0\right)$, $\alpha=\frac{c_2-c_1^2}{c_1}>0$ and $\U$ has free-binomial distribution,  $\beta(\sigma,\theta)$.
\end{theorem}

\begin{proof}
For any positive integer $n$ multiply both sides of \eqref{reg1} and \eqref{reg2} by $\left(\V^{\frac{1}{2}}\,\U\,\V^{\frac{1}{2}}\right)^n$ and take expectations $\varphi$. Therefore, by \eqref{ce1} of  Lemma \ref{conexp} we obtain, respectively,
\be\label{var1}
\varphi(\V(\V\U)^n)-\varphi((\V\U)^{n+1})=c_1\varphi((\V\U)^n)
\ee
and
\be\label{var2}
\varphi(\V^2(\V\U)^n)-2\varphi(\V(\V\U)^{n+1})+\varphi((\V\U)^{n+2})=c_2\varphi((\V\U)^n).
\ee

Introduce three sequences of numbers $(\alpha_n)_{n\ge 0}$, $(\beta_n)_{n\ge 0}$ and $(\gamma_n)_{n\ge 0}$ as follows
$$
\alpha_n=\varphi((\V\U)^n),\qquad \beta_n=\varphi(\V(\V\U)^n),\qquad\mbox{and}\qquad \gamma_n=\varphi(\V^2(\V\U)^n),\quad n=0,1,\ldots.
$$

Then equations \eqref{var1} and \eqref{var2} can be rewritten as
\be\label{rew1}
\beta_n-\alpha_{n+1}=c_1\alpha_n
\ee
and
\be\label{rew2}
\gamma_n-2\beta_{n+1}+\alpha_{n+2}=c_2\alpha_n.
\ee

Multiplying \eqref{rew1} and \eqref{rew2} by $z^n$, $z\in\C$, and summing up with respect to $n=0,1,\ldots$, we obtain the equations
\be\label{equ1}
B(z)-\frac{1}{z}(A(z)-\alpha_0)=c_1A(z)
\ee
and
\be\label{equ2}
C(z)-\frac{2}{z}(B(z)-\beta_0)+\frac{1}{z^2}(A(z)-\alpha_1z-\alpha_0)=c_2A(z),
\ee
where
$$
A(z)=\sum_{n=0}^{\infty}\,\alpha_n\,z^n,\qquad B(z)=\sum_{n=0}^{\infty}\,\beta_n\,z^n,\qquad C(z)=\sum_{n=0}^{\infty}\,\gamma_n\,z^n
$$
and the above series converge at least in some neighbourhood of zero, due to the fact that supports of $\U$ and $\V$ are compact. Note also, that since $\alpha_0=1$, we have $A=M_{\V\U}+1$.

Before we proceed further with equations \eqref{equ1} and \eqref{equ2} we need to establish some useful relations between sequences $(\alpha_n)$, $(\beta_n)$ and $(\gamma_n)$. To this end we need to define additional sequence $(\delta_n)_{n\ge 0}$, by setting
$$\delta_n=\varphi(\U(\V\U)^n),\qquad n=0,1,\ldots$$

Consider first the sequence $(\alpha_n)$. Note that by formula \eqref{BLS} it follows that
\begin{align*}
\alpha_n=&\Rr_1\varphi(\U(\U\V)^{n-1})
\\
&+\Rr_2[\varphi(\U)\varphi(\U(\U\V)^{n-2})+\varphi(\U\V\U)\varphi(\U(\V\U)^{n-3})+\ldots+\varphi(\U(\V\U)^{n-2})\varphi(\U)]\\
&+\ldots\\
&+\Rr_n\varphi^n(\U),
\end{align*}
where $\Rr_n=\Rr_n(\V)$ is the $n$th cumulant of the variable $\V$. Therefore, in terms of $\delta_n$'s we obtain
$$
\alpha_n=\Rr_1\delta_{n-1}+\Rr_2(\delta_0\delta_{n-2}+\delta_1\delta_{n-3}+\ldots+\delta_{n-2}\delta_0)+\ldots+\Rr_n\delta_0^n
$$
and thus for any $n=1,2,\ldots$
$$
\alpha_n=\sum_{k=1}^n\,\Rr_k\sum_{i_1+\ldots+i_k=n-k}\,\delta_{i_1}\ldots\delta_{i_k}.
$$
Consequently,
\begin{align*}
A(z)=&1+\sum_{n=1}^{\infty}\,z^n\,\sum_{k=1}^n\,\Rr_k\sum_{i_1+\ldots+i_k=n-k}\,\delta_{i_1}\ldots\delta_{i_k}
=1+\sum_{n=1}^{\infty}\,\sum_{k=1}^n\,\Rr_k\,z^k\sum_{i_1+\ldots+i_k=n-k}\,\delta_{i_1}z^{i_1}\ldots\delta_{i_k}z^{i_k}
\\
=&1+\sum_{k=1}^{\infty}\Rr_k\,z^k\sum_{n=k}^{\infty}\,\sum_{i_1+\ldots+i_k=n-k}\,\delta_{i_1}z^{i_1}\ldots\delta_{i_k}z^{i_k}
=1+\sum_{k=1}^{\infty}\Rr_k\,z^k\sum_{m=0}^{\infty}\,\sum_{i_1+\ldots+i_k=m}\,\delta_{i_1}z^{i_1}\ldots\delta_{i_k}z^{i_k}
\\=&1+\sum_{k=1}^{\infty}\,\Rr_k\,z^k\,\left(\sum_{i=0}^{\infty}\,\delta_iz^i\right)^k.
\end{align*}
Therefore
$$
A(z)=1+\sum_{k=1}^{\infty}\,\Rr_k\,(zD(z))^k,
$$
where $D$ is the generating function of the sequence $(\delta_n)$, that is $D(z)=\sum_{i=0}^{\infty}\,\delta_iz^i$.
Finally, with $r$ being the $r$-transform of $\V$, that is $r(z)=\sum_{k=0}^{\infty}\,R_{k+1}z^k$ we obtain
\be\label{AA}
A(z)=1+zD(z)r(zD(z)).
\ee

Similarly, by \eqref{BLS},
$$
\beta_n=\Rr_1\alpha_n+\Rr_2(\alpha_0\delta_{n-1}+\alpha_1\delta_{n-2}+\ldots+\alpha_{n-1}\delta_0)+\ldots+\Rr_{n+1}\alpha_0\delta_0^n
$$
and thus for any $n=0,1,\ldots$
$$
\beta_n=\sum_{k=1}^{n+1}\,\Rr_k\sum_{i_1+\ldots+i_k=n-k+1}\,\alpha_{i_1}\delta_{i_2}\ldots\delta_{i_k}.
$$
Therefore,
\begin{align*}
B(z)=&\sum_{n=0}^{\infty}\,\sum_{k=1}^{n+1}\,\Rr_k\sum_{i_1+\ldots+i_k=n-k+1}\,\alpha_{i_1}\delta_{i_2}\ldots\delta_{i_k}
\\
=&\sum_{n=0}^{\infty}\,\sum_{k=1}^{n+1}\,\Rr_k\,z^{k-1}\sum_{i_1+\ldots+i_k=n-k+1}\,\alpha_{i_1}z^{i_1}\delta_{i_2}z^{i_2}\ldots\delta_{i_k}z^{i_k}
\\
=&\sum_{k=1}^{\infty}\,\Rr_k\,z^{k-1}\sum_{n=k-1}^{\infty}\,\sum_{i_1+\ldots+i_k=n-k+1}\,\alpha_{i_1}z^{i_1}\delta_{i_2}z^{i_2}\ldots\delta_{i_k}z^{i_k}
\\
=&\sum_{k=1}^{\infty}\,\Rr_k\,z^{k-1}\sum_{m=0}^{\infty}\,\sum_{i_1+\ldots+i_k=m}\,\alpha_{i_1}z^{i_1}\delta_{i_2}z^{i_2}\ldots\delta_{i_k}z^{i_k}
\\
=&\sum_{k=1}^{\infty}\,\Rr_k\,z^{k-1}\,A(z)\,D^{k-1}(z)=A(z)r(zD(z)).
\end{align*}
Finally, applying \eqref{AA} we get
\be\label{BB}
B(z)=zD(z)r^2(zD(z))+r(zD(z)).
\ee

The formula for $\gamma_n$ is again based on \eqref{BLS}
\begin{align*}
\gamma_n&=\Rr_1\beta_n+\Rr_2(\alpha_n+\beta_0\delta_{n-1}+\beta_1\delta_{n-2}+\ldots+\beta_n\delta_0)
\\
&+\Rr_3[(\alpha_0\delta_{n-1}+\ldots+\alpha_{n-1}\delta_0)+(\beta_0\delta_0\delta_{n-2}+\ldots+\beta_{n-2}\delta_0^2)]+\ldots+\Rr_{n+2}\alpha_0\delta_0^n
\end{align*}
and thus it splits in two parts for any $n=0,1,\ldots$
$$
\gamma_n=\sum_{k=2}^{n+2}\,\Rr_k\sum_{i_1+\ldots+i_{k-1}=n-k+2}\,\alpha_{i_1}\delta_{i_2}\ldots\delta_{i_{k-1}}+\sum_{k=1}^{n+1}\,\Rr_k\sum_{i_1+\ldots+i_k=n-k+1}\,\beta_{i_1}\delta_{i_2}\ldots\delta_{i_k}.
$$
Therefore, also $C(z)$ splits in two parts
$$
C(z)=C_1(z)+C_2(z),
$$
where
$$
C_1(z)=\sum_{n=0}^{\infty}\,z^n\sum_{k=2}^{n+2}\,\Rr_k\sum_{i_1+\ldots+i_{k-1}=n-k+2}\,\alpha_{i_1}\delta_{i_2}\ldots\delta_{i_{k-1}}
$$
and
$$
C_2(z)=\sum_{n=0}^{\infty}\,z^n\sum_{k=1}^{n+1}\,\Rr_k\sum_{i_1+\ldots+i_k=n-k+1}\,\beta_{i_1}\delta_{i_2}\ldots\delta_{i_k}.
$$
The expression for the second part, $C_2$, can be derived exactly in the same way as it was done for $B(z)$. The computation yields
\be\label{C2}
C_2(z)=B(z)r(zD(z))=zD(z)r^3(zD(z))+r^2(zD(z)).
\ee

For the first part the derivation is similar though little more complicated
\begin{align*}
C_1(z)&=\sum_{n=0}^{\infty}\,\sum_{k=2}^{n+2}\,\Rr_k\,z^{k-2}\sum_{i_1+\ldots+i_{k-1}=n-k+2}\,\alpha_{i_1}z^{i_1}\delta_{i_2}z^{i_2}\ldots\delta_{i_{k-1}}z^{i_{k-1}}
\\
&=\sum_{k=2}^{\infty}\,\Rr_k\,z^{k-2}\sum_{m=0}^{\infty}\,\sum_{i_1+\ldots+i_{k-1}=m}\,\alpha_{i_1}z^{i_1}\delta_{i_2}z^{i_2}\ldots\delta_{i_{k-1}}z^{i_{k-1}}
\\
&=\sum_{k=2}^{\infty}\,\Rr_k\,z^{k-2}\,A(z)D^{k-2}(z)=\frac{A(z)}{zD(z)}\,(r(zD(z))-\Rr_1).
\end{align*}
Recalling \eqref{AA} we get
\be\label{C1}
C_1(z)=r(zD(z))[r(zD(z))-\Rr_1]+\frac{r(zD(z))-\Rr_1}{zD(z)}.
\ee

Finally, \eqref{C2} and \eqref{C1} together with $\Rr_1=\beta_0$ give
\be\label{CC}
C(z)=zD(z)r^3(zD(z))+r^2(zD(z))+r(zD(z))[r(zD(z))-\beta_0]+\frac{r(zD(z))-\beta_0}{zD(z)}
\ee

Now we can return to the system of equations \eqref{equ1} and \eqref{equ2}. Plugging expressions \eqref{AA} and \eqref{BB} into \eqref{equ1} we get
\be\label{pom}
zD(z)r^2(zD(z))+(1-D(z))r(zD(z))=c_1(1+zD(z)r(zD(z))).
\ee
Define a new function $h$
\begin{align*}
h(z)=zD(z)r(zD(z))
\end{align*}
note that $h=A-1=M_{\V\U}.$
\newline Multiply both sides of the above equation by $zD(z)$. Then \eqref{pom} can be written as
\be\label{equ1h}
h^2(z)=[(1+c_1z)D(z)-1]h(z)+c_1zD(z).
\ee
Therefore,
$$
h^3(z)=h(z)\,\left([(1+c_1z)D(z)-1]h(z)+c_1zD(z)\right)
$$
and thus
\be\label{equ1hh}
h^3(z)=\left(\left[(1+c_1z)D(z)-1\right]^2+c_1zD(z)\right)\,h(z)+[(1+c_1z)D(z)-1]\,c_1zD(z).
\ee

Similarly, plugging \eqref{AA}, \eqref{BB} and \eqref{CC} into \eqref{equ2} we obtain
\begin{align*}
zD(z)r^3(zD(z))+2r^2(zD(z))-\beta_0r(zD(z))+\frac{r(zD(z))-\beta_0}{zD(z)}+&\\
-\frac{2}{z}\left[zD(z)r^2(zD(z))+r(zD(z))-\beta_0\right]+\frac{1}{z^2}\left[zD(z)r(zD(z))-\alpha_1z\right]&=c_2[zD(z)r(zD(z))+1].
\end{align*}
In terms of $h$ the above equation reads
\begin{align}\label{equh3}
h^3(z)+2(1-D(z))h^2(z)-[\beta_0zD(z)-1+2D(z)-D^2(z)+c_2z^2D^2(z)]h(z)
=&c_2z^2D^2(z)\nonumber
\\
&+\beta_0zD(z)(1-2D(z))\nonumber
\\
&+\alpha_1zD^2(z).
\end{align}
Inserting \eqref{equ1hh} and \eqref{equ1h} into \eqref{equh3} after cancelations we get
\be
\label{funH}
h(z)=\frac{\lambda\alpha D(z)}{c_1\alpha z D(z)+\lambda\alpha-c_1}-1,
\ee
where
$$\alpha=\frac{c_1^2-c_2}{\alpha_1-\beta_0}=\frac{c_2-c_1^2}{c_1},\qquad\lambda=\frac{c_1(\alpha_1+c_1-2\beta_0)}{c_1^2-c_2}=
\frac{c_1^2}{c_2-c_1^2}+\frac{c_1(\alpha_1+2c_1-2\beta_0)}{c_1^2-c_2}=\theta+\sigma.$$
Plugging \eqref{funH} into \eqref{equ1h}, after canceling $D$ (which is allowed at least in a neighbourhood of zero, since $\delta_0>0$) we obtain the following quadratic equation for $D$:
\begin{align*}
\alpha\left(1+c_1\left(1-\frac{1}{\lambda}\right)z\right)zD^2(z)-
\left\{1+\left(\alpha(1-\lambda)+c_1\left(1-\frac{2}{\lambda}\right)\right)z\right\}D(z)
+1-\frac{c_1}{\alpha\lambda}=0.
\end{align*}
We want to express $h(z)$ as a function of $zD(z)$. To this end we write the above equation as $a_1z^2D^2(z)+a_2zD^2(z)+a_3zD(z)+a_4D(z)+a_5=0$, where $a_1,\ldots,a_5$, are suitable coefficients.  Then we substitute  one $D(z)$ in the second term and $D(z)$ in the  fourth term at the right hand side by $$D(z)=\frac{1}{\alpha\lambda}(h(z)+1)(\alpha\lambda + c_1\alpha zD(z)).$$
Note that the last identity is a consequence of \eqref{funH}.\\
After canceling $\alpha c_1zD(z)+\alpha_1$, which is allowed at least in the neighbourhood of zero since $\lim_{z\to 0}\,zD(z)=0$ and $\alpha_1>0$, we obtain,
\begin{align}
\label{hD}
\frac{h(z)}{zD(z)}=\frac{\lambda\alpha}{1-\alpha z D(z)}.
\end{align}
Recall that $h(z)=zD(z)r\left(zD(z)\right)$. Since $r$ is analytic at zero and $\lim_{z\to 0}zD(z)=0$, we conclude that
\begin{align}
r(z)=\lambda\alpha\frac{1}{1-\alpha z}.
\label{equP}
\end{align}
Note that above equation defines the $r$-transform of the free-Poisson distribution with rate $\lambda$, and jump size $\alpha$.

It remains  to show that $\U$ has free-binomial distribution, which can be done through calculating $S$-transforms.

Combining \eqref{funH} and \eqref{hD} we obtain the quadratic equation
\begin{align*}
\alpha z h^2(z)-\left\{1+\left(c_1-\alpha(1+\lambda)\right)z\right\}h(z)
-\left(c_1-\alpha\lambda\right)z=0.
\end{align*}
Since $h=M_{\V\U}$, from the above equation for $\Psi_{\V\U}=h^{-1}$, we get,
\begin{align*}
\alpha \Psi_{\V\U}(z)z^2-\left\{1+\left(c_1-\alpha(1+\lambda)\right)\Psi_{\V\U}\right\}z
-\left(c_1-\alpha\lambda\right)\Psi_{\V\U}(z)=0,
\end{align*}
which implies
\begin{align*}
\Psi_{\V\U}(z)=\frac{z}{(1+z)(\alpha \lambda-c_1+\alpha z)}.
\end{align*}
Now we use equations \eqref{MSr} to find corresponding $S$-transform as
$$S_{\V\U}(z)=\frac{1}{\alpha\lambda-c_1+\alpha z}.$$
Moreover $S$-transform of $\V$ is
$$S_{\V}(z)=\frac{1}{\alpha\lambda+\alpha z}.$$
Since $\U$ and $\V$ are free by (\ref{Scon}) we arrive at
$$S_{\U}(z)=1+\frac{c_1}{\alpha\lambda-c_1+\alpha z}.$$
Now we use \eqref{Str} and \eqref{Crr} to find Cauchy transform for $\U$ as
$$G_{\U}(z)=\frac{1+\frac{c_1}{\alpha}-\lambda+\left(\lambda-2\right)z
+\sqrt{\left(1+\frac{c_1}{\alpha}-\lambda+\left(\lambda-2\right)z\right)^2-
4\left(1-\lambda\right)z(z-1)}}
{2z(1-z)}.$$
From (\ref{betaC}) it follows that $G_{\U}$ is the Cauchy transform of free-binomial distribution with parameters $\sigma,\theta$, ($\lambda-\frac{c_1}{\alpha}=\sigma$, $\frac{c_1}{\alpha}=\theta$).

Finally, let us mention that a $W^*$-probability space with free random variables $\V$ and $\U$ with, respectively, free-Poisson and free-binomial distributions, can be constructed in a standard way as a free product of two $W^*$-probability spaces, one containing free-Poisson random variable, second with free-binomial distribution. For details see \cite{HaagerupLarsen}, \cite{VoiDykNica}.
\end{proof}
Combining Theorems \ref{direc} and \ref{main} we get equivalence, as in the classical situation.
\begin{corollary}
Let $\U,\,\V$ be free random variables in a $W^*$ probability space. Then $\V-\V^{1/2}\U\V^{1/2}$ and $\V^{1/2}\U\V^{1/2}$ are free if and only if $\V$ has free-Poisson distribution and $\U$ has free-binomial distribution.
\end{corollary}
\begin{proof}
The "if" part is trivially read out from theorem \ref{direc}. Since freeness implies that conditional moments are constant (see \eqref{ce2} in Lemma \ref{conexp}) the "only if" part follows from Theorem \ref{main}.
\end{proof}
\subsection*{Acknowledgement} The authors thank M. Bo\.zejko and W. Bryc for helpful comments and discussions. They are also grateful to the referee for a careful reading of the manuscript, in particular,  for remarks concerning final stages of the proof of Th. \ref{main}.
\bibliographystyle{plain}
\bibliography{Bibl}

\end{document}